\documentclass[12pt]{amsart}
\usepackage{a4}
\usepackage{amsthm, amsfonts, amssymb,latexsym}
\usepackage{enumerate, color, comment}
\usepackage[english]{babel}%[english]
\usepackage[latin1]{inputenc}

\newtheorem{theorem}{Theorem}[section]
\newtheorem{cor}[theorem]{Corollary}
\newtheorem{lemma}[theorem]{Lemma}

\def\R{\mathbb R}

\numberwithin{equation}{section}

\newcommand{\beq}[1]{\begin{equation}\label{#1}}
\newcommand{\eeq}{\end{equation}}

\title[If $A+A$ is small then $AAA$ is superquadratic]{If $A+A$ is small then $AAA$ is superquadratic}

\author[O. Roche-Newton and I. D. Shkredov]{Oliver Roche-Newton and Ilya D. Shkredov}

\address{O. Roche-Newton: Johann Radon Institute for Computational and Applied Mathematics (RICAM), 69 Altenberger Stra{\ss}e, Linz, Austria }
\email{o.rochenewton@gmail.com }

\address{I.D.~Shkredov: Steklov Mathematical Institute,ul. Gubkina, 8, Moscow, Russia, 119991
and 
IITP RAS, 
Bolshoy Karetny per. 19, Moscow, Russia, 127994 
and 
MIPT, 
Institutskii per. 9, Dolgoprudnii, Russia, 141701}
\email{ilya.shkredov@gmail.com }

\begin{document}

\begin{abstract}
This note proves that there exists positive constants $c_1$ and $c_2$ such that for all finite $A \subset \mathbb R$ with $|A+A| \leq |A|^{1+c_1}$ we have $|AAA| \gg |A|^{2+c_2}$.
\end{abstract} 
\maketitle

\section*{Notation}

Throughout the paper, the standard notation
$\ll,\gg$ is applied to positive quantities in the usual way. That is, $X\gg Y$ and $Y \ll X$ both mean that $X\geq cY$, for some absolute constant $c>0$.  The expression $X \approx Y$ means that both $X \gg Y$ and $X \ll Y$ hold. The notation $\lesssim$ and $\gtrsim$ is used to suppress both constant and logarithmic factors. To be precise, the expression $X\gtrsim Y$ or $Y \lesssim X$ means that $X\gg Y/(\log X)^c$, for some absolute constant $c>0$.  All logarithms have base $2$.

For $A \subset \mathbb R$, the sumset of $A$ is the set
$A+A:= \{a+b :a,b \in A\}$.
The difference set $A-A$ and the product set $AA$ are defined similarly. The $k$-fold sum set is $kA:= \{ a_1+ \dots + a_k : a_1, \dots, a_k \in A\}$.

\section{Introduction}

Let $A$ be a finite set of real numbers such that $|A+A|\leq|A|^{1+c}$ where $c$ is a small but positive real number.
%\footnote{We are mainly interested here in the case when $K=|A|^{c}$ for some small positive $c$.}
 According to sum-product phenomena, we expect that the product set $AA$ should be rather large. In fact, this ``few sums implies many products" problem is well understood. For example, Elekes and Ruzsa \cite{ER} used the Szemer\'{e}di-Trotter Theorem to prove the bound
\begin{equation}
  |A+A|^4|AA| \gg \frac{|A|^6}{\log |A|}, 
\label{er}
\end{equation}
which in particular gives
\begin{equation}
 |A+A| \ll |A|^{1+c} \Rightarrow |AA| \gtrsim |A|^{2-4c}.
\label{fsmp}
\end{equation}
A better dependence on $c$ can be obtained by using Solymosi's \cite{S} sum-product estimate
\begin{equation}
  |A+A|^2|AA| \gg \frac{|A|^4}{\log |A|}.
\label{soly}
\end{equation}
 %The key point about \eqref{fsmp} is that it gives a strong conclusion when $K=|A|^c$ for some small positive constant $c$.

One may consider the same question with multifold product sets. To be precise we expect that there exists $c>0$ such that the following weak form of the $k$-fold Erd\H{o}s-Szemer\'{e}di conjecture holds:
\begin{equation}
|A+A| \leq|A|^{1+c} \Rightarrow |A^{(k)}| \gtrsim |A|^k,
\label{fsmip}
\end{equation}
where $A^{(k)}$ denotes the $k$-fold product set $\{a_1  a_2 \cdots a_k : a_1, \dots, a_k \in A \}$. However, it seems that this problem is not well understood at all, and we are not aware even of a bound of the form
\begin{equation}
|A+A| \leq |A|^{1+c_1} \Rightarrow |A^{(100)}| \gg |A|^{2+c_2},
\label{fsmip2}
\end{equation}
for some positive constants $c_1$ and $c_2$. The aim of this note is to prove such a superquadratic bound.

\begin{theorem} \label{thm:main} There exist positive constants $c_1$ and $c_2$ such that for all sufficiently large finite sets $A \subset \mathbb R$ with $|A+A|\leq |A|^{1+c_1}$, it follows that $|AAA| \geq |A|^{2+c_2}$.
\end{theorem}

We will give two slightly different proofs of this result. The first proof uses a result of Shkredov and Zhelezov \cite{SZ}, and may be preferable for the reader familiar with \cite{SZ}. The second is more self-contained and is presented in a quantitative form, although we do not pursue the best possible bounds given by our methods, instead preferring to simplify the proof at certain points. The forthcoming Theorem \ref{thm:mainprime} gives a quantitative formulation of Theorem \ref{thm:main} with $c_2=\frac{1}{392}-c_1\frac{125}{56}-o(1)$.

The two proofs are certainly similar, but we feel that they are different enough, and short enough, to warrant giving both in full. On common feature is that they both rely on a threshold-breaking bound for the additive energy of a set with small product set. The additive energy of $A$, denoted $E^+(A)$, is the number of solutions to the equation
\[ 
a_1+a_2=a_3+a_4, \,\,\,\,\,\,\,\ a_i \in A.
\]
A simple application of the Cauchy-Schwarz inequality gives the much used bound
\begin{equation}
E^+(A) \geq \frac{|A|^4}{|A+A|}.
\label{CSclassic}
\end{equation}

We will use the following result, which is \cite[Theorem 3]{MRSS}. 

\begin{theorem} \label{thm:energyilya}
	Let $A\subset \R$ and $|AA| \le M|A|$. 
	Then 
\[
	E^{+} (A) \ll M^{8/5} |A|^{49/20} \log^{1/5} |A| \,.
\]
\end{theorem}

The proof of Theorem \ref{thm:energyilya} uses the Szemer\'{e}di-Trotter Theorem, as well as higher energy tools, which have been used to push past several thresholds for sum-product type problems in recent years. Note that a weaker bound of the form $E^+(A) \ll_M |A|^{5/2}$ can be obtained by a more straightforward application of the Szemer\'{e}di-Trotter Theorem, and is implicitly contained in the work of Elekes \cite{E}. However, for both proofs of the main theorem of this paper, it is crucial that the exponent $49/20$ is strictly less than this threshold of $5/2$.

Finally, we will use the arguments of the second proof to establish the following result.

\begin{theorem} \label{thm:triplesumproduct}
Let $A \subset \mathbb R$ be finite. Then
\[
|(A+A)(A+A)(A+A)| \gtrsim |A|^{2+\frac{1}{392}}.
\]
\end{theorem}
The bound
\begin{equation} \label{minkowski}
|(A\pm A)(A\pm A)| \gg \frac{|A|^2}{\log |A|}
\end{equation}
was established in \cite{RNR}. The question of whether the growth of the sum or difference sets continues with more products was considered in \cite{BRNZ}, where ideas from \cite{S} were used to prove that
\begin{equation} \label{proddiff}
|(A-A)(A-A)(A-A)| \gtrsim |A|^{2+\frac{1}{8}}.
\end{equation}
However, the tricks used to prove \eqref{proddiff} are somewhat inflexible and do not allow for $A-A$ to be replaced with $A+A$. Theorem \ref{thm:triplesumproduct} gives the first superquadratic bound for $(A+A)(A+A)(A+A)$.

\subsection{Other tools}

Some other well-known results that are used in our proofs are collected here. The following two forms of the Pl\"{u}nnecke-Ruzsa inequality are applied. See Petridis \cite{P} for short proofs.

\begin{lemma} \label{lem:Plun1} Let $X$ be a finite set in an additive abelian group. Then
\[ |kX-lX| \leq \frac{|X+X|^{k+l}}{|X|^{k+l-1}}.\]
\end{lemma}

\begin{lemma} \label{lem:Plun2}

Let $X$ and $Y$ be finite subsets in an additive abelian group. Then
\[
|kX| \leq \frac{|X+Y|^k}{|Y|^{k-1}}.
\]

\end{lemma}

The ratio set of $A$ is the set $A/A=\{a/b : a,b \in A \}$. We need an estimate for the ratio set when the sum set is small. Solymosi's estimate \eqref{soly} also holds when $AA$ is replaced with $A/A$. The following variant, which removes the logarithmic factor, was observed by Li and Shen \cite{LS}: for any finite set $A \subset \mathbb R$, 
\begin{equation}
|A+A|^2|A/A| \gg |A|^4.
\label{soly2}
\end{equation}

Given finite sets $A, B \in \mathbb R$, let $T(A,A,B)$ denote the number of solutions to the equation
\[ \frac{a_1-b}{a_2-b} = \frac{a_1'-b'}{a_2'-b'} ,\,\,\,\,\,\,\, a_1,a_2,a_1',a_2' \in A, b,b' \in B.
\]
The notation $T(A)$ is used as shorthand for $T(A,A,A)$. This is essentially the number of collinear triples in the point set $A \times A$, an observation which Jones \cite{J} used to prove that
\begin{equation}
T(A) \ll |A|^4 \log |A|.
\label{triples}
\end{equation}

\section{The first proof}

Given finite sets $B, C$ and $x \in \mathbb R$, the notation $r_{B+C}(x)$ is used for the number of representations of $x$ as an element of $B+C$. That is,
\[
r_{B+C}(x):= |\{(b,c) \in B \times C : b+c=x\}|.
\]

In our first proof of Theorem \ref{thm:main}, we will need the following lemma, which is essentially contained in \cite{SZ} (see Remark 1 therein). We omit the proof, since all of the necessary details can be found in \cite{SZ}, but remark that a crucial ingredient is a threshold-breaking bound for the additive energy of a similar form to Theorem \ref{thm:energyilya}.

For sets $B,C,X \subset \mathbb R$, define
\[
\sigma_X(B,C):= \sum_{x \in X} r_{B+C}(x),
\]
and let $\sigma_X(B)=\sigma_X(B,C)$.

\begin{lemma} \label{ID} There exist positive constants $c$ and $c'$ such that the following holds. For any finite $X \subset \mathbb R$ such that $|XX| \leq |X|^{1+c}$, and for any finite $B \subset \mathbb R$,
%IS. May be we should write |B|^{5/6} and |C|^{13/15}, no? 
%ORN: Yes you are right. I simplified further and just bounded everything by the larger set, since that is what we do anyway in the application. This also makes the proof easier to copy from your paper with Dmitry, since we can simply use the symmetric case as $ \sigma_X(B,C) \leq \sigma_X(B \cup C, B \cup C)$.
\[ \sigma_X(B) \lesssim |B|^{\frac{17}{10}} |X|^{\frac{3}{20} -c'}. \]

\end{lemma}

We will first prove a slightly weaker version of Theorem \ref{thm:main}, using the $4$-fold rather than $3$-fold product set.

\begin{theorem} \label{thm:4fold} There exist positive constants $c_1$ and $c_2$ such that for all sufficiently large finite sets $A \subset \mathbb R$ with $|A+A|\leq |A|^{1+c_1}$, it follows that $|AAAA| \geq |A|^{2+c_2}$.
\end{theorem}

\begin{proof}

Let $c$ and $c'$ be the positive constants given by Lemma \ref{ID}. The constant $c_1$ must be sufficiently small compared to $c$ and $c'$: taking $c_1:= \frac{1}{4}\min \{c,c'\}$ would comfortably suffice. We can take $c_2=c$.

Suppose that $|A+A| \leq |A|^{1+c_1}$. Recall from \eqref{soly} that $|AA| \gtrsim |A|^{2-2c_1}$. It would be sufficient to show that $|(AA)(AA)| \geq |AA|^{1+c}$, as we would then have
\[ |AAAA| \geq |AA|^{1+c} \gtrsim |A|^{(1+c)(2-2c_1)} \geq |A|^{2+c_2}.
\]

Suppose for a contradiction that this is not true and we have $|(AA)(AA)| \leq |AA|^{1+c}$. Note that, for any $a \in A$ we have $A \subset a^{-1}AA$. Define $X:= a^{-1}AA$. By our assumption, $|XX| \leq |X|^{1+c}$. Write $S:= A+A$. Then, notice that
\[\sigma_X(-A,S)=\sum_{x \in X} r_{S-A}(x) \geq  \sum_{x \in A} r_{S-A}(x) \geq |A|^2, \]
since $r_{S-A}(x) \geq |A|$ for all $x \in A$ (because of the solutions $x=(x+y)-y$). On the other hand, by Lemma \ref{ID},
%IS Why 34/20? One of the sets A not S, so it should be |A|^{26/30} and |S|^{25/30}, right?
\[ |A|^2 \leq  \sigma_X(-A,S) \leq \sigma_X(S \cup -A) \lesssim  |S|^{17/10}|X|^{3/20-c'} . \]
With our earlier choice of $c_1$, this implies that $|X| \geq |A|^{2+\epsilon}$ for some constant $\epsilon >0$, which is a contradiction.

\end{proof}

It is then a straightforward task to use the Pl\"{u}nnecke-Ruzsa Theorem to complete the proof of Theorem  \ref{thm:main}.

\begin{proof}[Proof of Theorem \ref{thm:main}]
Let $c_1$ be a sufficiently small positive constant (no larger than the $c_1$ in the statement of Theorem \ref{thm:4fold}). So, $|AAAA| \geq |A|^{2+c_2}$, where $c_2$ is the constant from the statement of Theorem \ref{thm:4fold}.  Applying Lemma \ref{lem:Plun2} with $X=A$ and $Y=AA$ gives
\[
|AAAA| \leq \frac{|(AA)A|^4}{|AA|^3}
\]
Also, $|AA| \gtrsim |A|^{2-2c_1}$ by \eqref{soly}. Putting all of this together gives
\[|AAA| \gtrsim |A|^{2 + \frac{1}{4}(c_2-6c_1)}. \]
Choosing $c_1$ sufficiently small gives $|AAA| \gtrsim |A|^{2+\frac{c_2}{8}}$. This completes the proof.

\end{proof}

\section{The second proof}

Similarly to the previous section, we will first prove a superquadratic bound involving more variables.

\begin{lemma} \label{lem:8prod} Let $A \subset \mathbb R$ be finite and write $|A+A|=K|A|$. Then
\[
\left | \frac{AAAA}{AAAA} \right | \gtrsim \frac{ |A|^{\frac{100}{49}}}{K^{\frac{40}{7}}}.
\]
\end{lemma}

\begin{proof}

Let $S:=A+A$. 
%Without loss of generality let us assume that $0\in A$. 
\begin{comment}
We start with the formula 
\begin{equation}\label{f:basic}
	1- \frac{b-c}{b'-c} = \frac{b'-c'}{b'-c} \left( 1- \frac{b-c'}{b'-c'} \right) 
	%\,.
\end{equation}
\end{comment}
Consider the set
\[
	\mathcal{A} := \{(b,b',c,c') \in S \times S \times A \times A ~:~ b-c, b-c',b'-c,b'-c' \in A \} \,.
\] 
By the Cauchy-Schwarz inequality,
\begin{align*}
|\mathcal{A}| = \sum_{c,c' \in A} \left( \sum_{b\in S} A(b-c) A(b-c') \right)^2 & \ge |A|^{-2} \left( \sum_{c,c'\in A} \sum_{b \in S} A(b-c) A(b-c')  \right)^2 
	\\&= E^{+} (A)^2 |A|^{-2} \stackrel{\eqref{CSclassic}}{\ge} |A|^6 |S|^{-2} := X \,.
\end{align*}
Write $|\mathcal{A} (b,b'c)| := \sum_{c'} \mathcal{A} (b,b',c,c')$.  
In these terms 
\begin{align*}
 X \leq |\mathcal A|&= \sum_{(b,b',c) \in S \times S \times A} |\mathcal{A} (b,b'c)|
 \\& = \sum_{(b,b',c) \in S \times S \times A : |\mathcal A(b,b',c)| \leq \frac{X}{2|A||S|^2}} |\mathcal A(b,b',c)|+  \sum_{(b,b',c) \in S \times S \times A : |\mathcal A(b,b',c)| > \frac{X}{2|A||S|^2}} |\mathcal{A} (b,b'c)|
 \\& \leq \frac{X}{2} +   \sum_{(b,b',c) \in S \times S \times A : |\mathcal A(b,b',c)| > \frac{X}{2|A||S|^2}} |\mathcal{A} (b,b'c)|.
\end{align*}
This implies that
\begin{equation}
\frac{X}{2} \le  \sum_{b,b',c ~:~ |\mathcal{A} (b,b',c)| > X/(2|A||S|^2)} |\mathcal{A} (b,b',c)| \le |A| \sum_{b,b',c ~:~ |\mathcal{A} (b,b',c)| >  X/(2|A||S|^2)} 1 .
\label{CSsetup}
\end{equation}
Define
\[
n(x):= \left |  \left\{ (b,b',c) \in S \times S \times A : x= \frac{b-c}{b'-c}, |\mathcal A(b,b',c)| > \frac{X}{2|A||S|^2}  \right \} \right|.
\]
Inequality \eqref{CSsetup} gives $\sum_x n(x) \geq \frac{X}{2|A|}$. Therefore, by the Cauchy--Schwarz inequality 
\begin{align*} 
	X^2/(2|A|)^2 &\le \sum_x n^2(x) \cdot \left| \left\{ \frac{b-c}{b'-c} ~:~  |\mathcal{A} (c,b',b)| \ge X/(2|A||S|^2)  \right\} \right|
	\\&  \leq 
		T(S,S,A) \cdot 
	\left| \left\{ \frac{b-c}{b'-c} ~:~  |\mathcal{A} (c,b',b)| \ge X/(2|A||S|^2)  \right\} \right| 
	\\& \ll
	|S|^4 \log |A| \cdot 
	\left| \left\{ \frac{b-c}{b'-c} ~:~  |\mathcal{A} (c,b',b)| \ge X/(2|A||S|^2)  \right\} \right| 
	\,.
\end{align*}
Here we have used the trivial bound $T(S,S,A) \leq T(S \cup A)$ and applied \eqref{triples}. Define $R$ to be the set on the right-hand side of the last inequality, so
\begin{equation}
|R| \gg \frac{X^2}{|A|^2|S|^4\log |A|}=\frac{|A|^{2}}{K^8 \log|A|}.
\label{Rbound}
\end{equation}

Now, consider the identity
\begin{equation}\label{f:basic}
	1- \frac{b-c}{b'-c} = \frac{b'-c'}{b'-c} - \frac{b-c'}{b'-c}  .
	%\,.
\end{equation}
From \eqref{Rbound} it 
follows that equation \eqref{f:basic} 
gives at least 
\[
	|R|\cdot X/(2|A||S|^2) \gg\frac{ |A|^{3}}{ K^{12} \log |A|} 
\]
solutions to the equation 
\begin{equation}\label{f:basic'}
1- \alpha_1  = \alpha_2 - \alpha_3 \,, \quad \quad \alpha_1,\alpha_2,\alpha_3 \in A/A \,.
\end{equation}
Indeed, for any $r \in R$ we fix a representation $r=\frac{b-c}{b'-c}$. There are at least $X/(2|A||S|^2) $ elements $c'$ such that $(b,b',c,c') \in \mathcal{A}$. As $c'$ varies, so do the elements $\frac{b'-c'}{b'-c}$ and $\frac{b-c'}{b'-c}$, and so the solutions to \eqref{f:basic'} obtained in this way are all distinct.

Further multiplying equation 
%the last equation 
\eqref{f:basic'} 
by any $a\in A/A$, we obtain that 
\[E^{+} (AA/AA) \gg  \frac{ |A|^{3}}{ K^{12} \log |A|}  |A/A| \stackrel{\eqref{soly2}}{\gg} \frac{|A|^5}{K^{14} \log|A|}.
\]
On the other hand, Theorem \ref{thm:energyilya} implies that
\[
E^+(AA/AA) \lesssim \left | \frac{AAAA}{AAAA} \right |^{\frac{49}{20}}.
\]
Combining the last two inequalities gives
\[
\left | \frac{AAAA}{AAAA} \right | \gtrsim \frac{ |A|^{\frac{100}{49}}}{K^{\frac{40}{7}}},
\]
as required

\end{proof}

We are now ready to prove the following quantitative form of Theorem \ref{thm:main}. 

\begin{theorem} \label{thm:mainprime}
Let $A \subset \mathbb R$ be finite and write $|A+A|=K|A|$. Then
\[
|AAA| \gtrsim \frac{|A|^{2+\frac{1}{392}}}{K^{\frac{125}{56}}}.
\]
\end{theorem}

\begin{proof} Applying Lemma \ref{lem:Plun1} in the multiplicative setting with $X=AA$ and $k=l=2$ gives
\begin{equation}
|AAAA|^4 \geq  |AA|^3 \left | \frac{AAAA}{AAAA} \right |.
\end{equation}
Applying Lemma \ref{lem:Plun2} with $X=A$ and $Y=AA$ gives
\begin{equation}
|AAAA| \leq \frac{|AAA|^4}{|AA|^3}.
\end{equation}
Combining these two inequalities with Lemma \ref{lem:8prod}, we have
\[
|AAA|^{16} \geq |AA|^{12} |AAAA|^4 \geq |AA|^{15}  \left | \frac{AAAA}{AAAA} \right | \gtrsim |AA|^{15} \cdot \frac{ |A|^{\frac{100}{49}}}{K^{\frac{40}{7}}} \stackrel{\eqref{soly}}{\gtrsim} \frac{|A|^{30}}{K^{30}} \cdot  \frac{ |A|^{\frac{100}{49}}}{K^{\frac{40}{7}}},
\]
and a rearrangement completes the proof.
\end{proof}

We can use the same ideas to prove Theorem \ref{thm:triplesumproduct}, although it is somewhat easier as the inital pigeonholing step is not required. We will make use of some known sum-product type results.

Define 
\[
R'[A]:=\left\{ \frac{b+c}{b'+c} : b,b',c \in A  \right \} .
\]
Then we have
\begin{equation} \label{r'}
|R'[A]| \gg \frac{|A|^2}{\log |A|}.
\end{equation}
For the set $R[A]:=\left\{ \frac{b-c}{b'-c} : b,b',c \in A  \right \} $, the bound $|R[A]| \gg \frac{|A|^2}{\log|A|}$ was established by Jones \cite{J} (see also \cite{RN} for an alternative presentation). The proof makes use of the bound \eqref{triples} on $T(A)$ and can be easily adapted by using the same bound for $T(A \cup -A)$ to give \eqref{r'}. We do not include a full proof.

We also make use of the main result of \cite{BRN}, which is that for any finite  $A \subset \mathbb R$,
\begin{equation} \label{eq:BRN}
\left| \frac{A+A}{A+A} \right | \gg |A|^2.
\end{equation}

\begin{proof}[Proof of Theorem \ref{thm:triplesumproduct}]

Again write $S=:A+A$. Fix an element $y \in R'[A]$ and fix its representation $y=\frac{b+c}{b+c'}$. Identity \eqref{f:basic} can be adapted to give
\[
1 - \frac{b+c}{b'+c} = \frac{b'+c'}{b'+c} - \frac{b+c'}{b'+c}.
\]
It follows that there are at least $|R'[A]||A| \stackrel{\eqref{r'}}{\gtrsim} |A|^3$ solutions to the equation
\[
1-x_1 = x_2 -x_3
\]
with $x_1,x_2,x_3 \in S/S$. Therefore, multiplying by any element of $S/S$ and applying \eqref{eq:BRN}, it follows that $E^+(SS/SS) \gtrsim |A|^5$. Applying Theorem \ref{thm:energyilya} %\footnote{As was the case in the second proof of Lemma \ref{lem:8prod}, we can obtain slightly better quantitative bounds by applying Theorem \ref{thm:energyilya} more carefully, but we do not pursue this improvement, prefering instead to simplify the proof slightly.}, 
then gives
\[
|A|^5 \lesssim E^+(SS/SS) \lesssim \left|\frac{SSSS}{SSSS} \right |^{49/20},
\]
and so
\begin{equation} \label{8products}
\left|\frac{SSSS}{SSSS} \right | \gtrsim |A|^{\frac{100}{49}}.
\end{equation}

Some applications of Pl\"{u}nnecke's inequalities complete the proof. Applying Lemma \ref{lem:Plun1} in the multiplicative setting with $X=SS$ and $k=l=2$ gives
\begin{equation}
|SSSS|^4 \geq  |SS|^3 \left | \frac{SSSS}{SSSS} \right |.
\end{equation}
Applying Lemma \ref{lem:Plun2} with $X=S$ and $Y=SS$ gives
\begin{equation}
|SSSS| \leq \frac{|SSS|^4}{|SS|^3}.
\end{equation}
Combining these two inequalities with \eqref{8products}, we have
\[
|SSS|^{16} \geq |SS|^{12} |SSSS|^4 \geq |SS|^{15}  \left | \frac{SSSS}{SSSS} \right | \gtrsim |SS|^{15} \cdot  |A|^{\frac{100}{49}} \stackrel{\eqref{minkowski}}{\gtrsim} |A|^{30} \cdot   |A|^{\frac{100}{49}},
\]
and a rearrangement completes the proof. \end{proof}

As was the case in the second proof of Lemma \ref{lem:8prod}, we can obtain slightly better quantitative bounds by applying Theorem \ref{thm:energyilya} more carefully, but we do not pursue this improvement, prefering instead to simplify the proof slightly

Finally, note that, if we instead apply Lemma \ref{lem:Plun2} to \eqref{8products} with $k=l=4$, we obtain the bound
\[ |S|^{\frac{50}{49}} \leq |A|^{\frac{100}{49}} \lesssim \left|\frac{SSSS}{SSSS} \right| \leq \frac{|SS|^8}{|S|^7}.\]
In particular, this gives $|SS| \gtrsim |S|^{1+\frac{1}{398}}$, and thus the sum set grows under multiplication. A similar, although quantitatively stronger, result for the difference set was given in \cite{S}.

\section*{Acknowledgements} 

Oliver Roche-Newton was partially supported by the Austrian Science Fund FWF Project P 30405-N32 as well as by the Austrian Science Fund (FWF), Project F5507-N26, which is a part of the Special Research Program Quasi-Monte Carlo Methods: Theory and Applications. We are grateful to George Shakan, Audie Warren and Dmitry Zhelezov for various helpful conversations and advice. We also thank the anonymous referee for helpful comments and corrections.


\begin{thebibliography}{99}

\bibitem{BRN} A. Balog and O. Roche-Newton, `New sum-product estimates for real and complex numbers', \textit{Discrete Comput. Geom.} 53 (2015), no. 4, 825-846.

\bibitem{BRNZ} A. Balog, O. Roche-Newton and D. Zhelezov, `Expanders with superquadratic growth', \textit{Electron. J. Combin.} 24 (2017), no. 3, Paper 3.14, 17pp.

%\bibitem{Chang} M. C. Chang, `Sums and products of different sets', \textit{Contrib. Discrete Math. 1}  no. 1 (electronic) (2006), 47--56.

\bibitem{E} G. Elekes, `On the number of sums and products', \textit{Acta Arith.} \textbf{81} (1997), 365-367.

%\bibitem{ENR} G. Elekes, M. Nathanson and I. Ruzsa, `Convexity and sumsets', \textit{J Number Theory.} \textbf{83} (1999), 194-201.

\bibitem{ER} G. Elekes and I. Ruzsa, `Few sums, many products', \textit{Studia Scientiarum Mathematicarum Hungarica} \textbf{40} (2003), 301-308.

%\bibitem{TJthesis} T. G. F. Jones `New quantitative estimates on the incidence geometry and growth of finite sets', PhD thesis, available at \textit{arXiv:1301.4853} (2013).

\bibitem{J} T. G. F. Jones `New results for the growth of sets of real numbers',  \textit{Discrete Comput. Geom.} \textbf{54} (2015), no. 4, 749-770.

%\bibitem{KS} S. Konyagin and I. Shkredov, `On sum sets of sets, having small product set', \textit{Proc. Steklov Inst. Math.} \textbf{290} (2015), 288-299.

%\bibitem{KS2} S. Konyagin and I. Shkredov, `New results on sums and products in $\mathbb R$', \textit{Proc. Steklov Inst. Math.} \textbf{294} (2016), 87-98.

\bibitem{LS} L. Li and J. Shen `A sum-division estimate of reals', \textit{Proc. Amer. Math. Soc.} \textbf{138} (2010), 101-104.

\bibitem{MRSS} B. Murphy, M. Rudnev,  I.D. Shkredov and Y.N. Shteinikov `On the few products, many sums problem', arXiv:1712.00410v1 [math.CO] 1 Dec 2017.

%\bibitem{MORNS} B. Murphy, O. Roche-Newton and I. Shkredov `Variations on the sum-product problem', \textit{SIAM J. Discrete Math.} 29 (2015), no. 1, 514-540.

%\bibitem{MORNS2} B. Murphy, O. Roche-Newton and I. Shkredov `Variations on the sum-product problem II', \textit{SIAM J. Discrete Math.} 31 (2017), no. 3, 1878-1894


\bibitem{P} G. Petridis, `New proofs of Pl\"{u}nnecke-type estimates for product sets in groups', \textit{Combinatorica} 32 (2012), no. 6, 721-733.

\bibitem{RN} O. Roche-Newton, `A short proof of a near-optimal cardinality estimate for the product of a sum set', \textit{Proceedings of 31st International Symposium on Computational Geometry} (2015), 74-80. 


\bibitem{RNR} O. Roche-Newton and M. Rudnev, `On the Minkowski distances and products of sum sets', \textit{Israel J. Math.} 209 (2015), no. 2, 507-526.

%\bibitem{SS} T. Schoen and I. Shkredov, `Higher moments of convolutions', \textit{J. Number Theory} 133 (2013), no. 5, 1693-1737.  

%\bibitem{Sh3} I. D.  Shkredov, `Some remarks on the asymmetric sum-product phenomenon', \textit{arXiv:1705.09703}(2017).

\bibitem{S} I. D.  Shkredov, `Difference sets are not multiplicatively closed', \textit{Discrete Anal.} (2016), Paper No. 17, 21 pp.

\bibitem{SZ} I. D.  Shkredov and D. Zhelezov, `On additive bases of sets with small product set', \textit{Int. Math. Res. Not. IMRN} (2018), no.5, 1585-1599.

\bibitem{S} J. Solymosi, `Bounding multiplicative energy by the sumset', \textit{Adv. Math.} \textbf{222} (2009), 402-408.

%\bibitem{solymosi2} J. Solymosi, `On the number of sums and products', \textit{Bull. London Math. Soc.} \textbf{37} (2005), 491-494.

%\bibitem{tv} T. Tao, V. Vu. 'Additive combinatorics' \textit{Cambridge University Press} (2006).









\end{thebibliography}
\end{document}